\documentclass{tran-l}
\usepackage[T1]{fontenc}
\usepackage{amssymb}
\usepackage{mathrsfs}
\usepackage{hyperref}
\hypersetup{hidelinks,pdftitle={A counterexample to the local existence of conformal normal metrics},pdfauthor={Zhe Zhang}}
\makeatletter
\AtBeginDocument{\let\@serieslogo\@empty}
\makeatother
\newtheorem{theorem}{Theorem}[section]
\newtheorem{lemma}[theorem]{Lemma}

\theoremstyle{definition}

\newtheorem{question}[theorem]{Question}
\theoremstyle{remark}
\newtheorem{remark}[theorem]{Remark}
\numberwithin{equation}{section}
\newcommand{\R}{\mathbb R}
\newcommand{\dd}{\,d}
\newcommand{\ip}[2]{\langle #1,#2\rangle}

\DeclareMathOperator{\supp}{supp}
\DeclareMathOperator{\dist}{dist}
\DeclareMathOperator{\divg}{div}

\begin{document}
\title[Local existence of conformal normal metrics]{A counterexample to the local existence of conformal normal metrics}
\author{Zhe Zhang}
\address{School of Arts and Sciences,
North China Institute of Aerospace Engineering,
Aimin East Road, Langfang, Hebei 065099, China}
\email{zhangzhe@nciae.edu.cn}
\date{}
\begin{abstract}
For every dimension $n\ge3$, we construct a smooth Riemannian metric whose entire Taylor series at a point agrees with the Euclidean metric and whose conformal class contains no smooth conformal normal metric near that point. For the metric used in the construction, parametrized by a smooth angular function, we prove that conformal normalization forces a dimension-dependent exponential decay rate for the spherical harmonic components of that function. This gives counterexamples whenever the angular coefficients decay too slowly. In particular, every nonanalytic angular function yields a counterexample. The proof combines a weighted cylinder estimate with an integral identity relating the angular spectral coefficients of the background volume defect to the nonlinear remainder. We also give a direct counterexample to the linear right inverse used in Cao's smooth existence argument and identify the failure of angular summability in that construction.\end{abstract}
\maketitle
\section{Introduction}\label{sec:introduction}
A Riemannian metric $g$ is \emph{conformal normal at a point $p$} if,
in $g$-normal coordinates at $p$,
\begin{equation}\label{densityone}
 \det(g_{ij})=1
\end{equation}
on a neighborhood of $p$. The local existence problem asks whether a given smooth metric $g_0$ has a smooth positive conformal factor $\Phi$
such that $g=\Phi g_0$ is conformal normal at $p$.  Lee and Parker \cite[Theorem~5.1]{LeeParker} proved that for each
prescribed finite order $N$, a smooth Riemannian metric can be
conformally changed so that its determinant in geodesic normal
coordinates is $1+O(|x|^N)$. G\"unther \cite{Gunther1986} established local conformal normalization for real-analytic metrics; see also~\cite[p.~173]{Gunther1993}. Cao \cite{Cao} asserted local
existence for arbitrary smooth metrics. The distinction between
jet-level normalization and normalization on a neighborhood is
essential: we construct smooth metrics for which the second assertion
fails even though the entire Taylor series is Euclidean.

Throughout, $n\ge3$, $S=S^{n-1}$ is the unit sphere,
$h_0$ is its round metric, and $d\omega$ is its volume measure.
We write $x=r\theta$ for Euclidean polar coordinates, with
$r=|x|$, and denote the Euclidean metric by $\delta$.
Our Laplacian convention is $\Delta=\divg\nabla$.
The operators with subscript $S$ refer to the round
metric unless another metric is indicated. Let $\Pi_\ell$ be
orthogonal projection in $L^2(S,d\omega)$ onto the spherical harmonics of degree $\ell$, and put $A_\ell=\Pi_\ell A$.

For a metric $g$, let $\rho$ denote its distance from $0$ on a normal
neighborhood.  The defining
condition \eqref{densityone} is equivalent to
\begin{equation}\label{cn}
 H_\rho=\frac{1}{n-1}\Delta_g\rho=\frac1\rho.
\end{equation}

\begin{theorem}\label{main}
Let $A\in C^\infty(S;\R)$. For $r>0$ sufficiently small, set
$s=-\ln r$ and define
\begin{equation}\label{metric}
 \begin{gathered}
 F(s)=\int_s^\infty e^{-\zeta^2}\dd\zeta,\qquad
 q(s,\theta)=F(s)A(\theta),\\
 \kappa(s,\theta)=e^{-2q(s,\theta)/(n-1)},\qquad
 g_A=dr^2+r^2\kappa^{-1}h_0.
 \end{gathered}
\end{equation}
The metric $g_A$ extends to a smooth Riemannian metric across $0$, with
the same entire Taylor series as $\delta$.
Suppose there exists a number $r_1>0$ and a smooth function $\Phi$ on
$B_{r_1}(0)$ such that $\Phi g_A$ is conformal normal at $0$ on this
ball. Then there exists a sufficiently small number $r_0<r_1$ such that
\begin{equation}\label{mainrate}
 \|A_\ell\|_{L^2(S)}\le Ce^{-c_n\ell}\quad(\ell\ge0),
 \qquad c_n=\frac1{16(n-1)},
\end{equation}
where $C>0$ depends on $n$, $r_0$, $\|A\|_{C^1}$, and
$\|\Phi\|_{C^1(\overline B_{r_0}(0))}$.
In particular, $A$ is real analytic on $S$.
\end{theorem}

Theorem~\ref{main} contradicts Cao's Corollary~B
\cite[p.~916]{Cao} when applied to the conformal-normal equation.
Indeed, the Cartesian coordinates $x=r\theta$ are $g_A$-normal
coordinates, and
\[
 \det((g_A)_{ij})-1=e^{2q}-1
\]
is smooth and flat at $0$, meaning that all its Cartesian derivatives
vanish there. Thus $g_A$ satisfies assumption~(A) in
\cite[p.~915]{Cao}. The right-hand side of his normalization equation
$(**)$ is
\[
 -\frac1r\partial_r\ln\sqrt{\det(g_A)}
 =-r^{-2}e^{-(-\ln r)^2}A(\theta),
\]
which is also smooth and flat.  Corollary~B asserts that this equation
has a smooth flat solution on a sufficiently small neighborhood.
By Cao's Theorem~A, such a solution gives a conformal normal
representative. For any smooth nonanalytic $A$, this contradicts
Theorem~\ref{main}.

We briefly describe the proof of Theorem~\ref{main}. In logarithmic radius, the background volume defect is
$B=e^{-s^2}A(\theta)$. A hypothetical conformal normal representative
determines a smooth distance quotient $w=\ln(\rho/r)$.
Subtracting a M\"obius comparison solution gives a smooth flat
remainder $v$, which satisfies
\[
 (N+R)v=-B-R_\beta,\qquad N=(n-1)\partial_s^2+\Delta_S.
\]
The coefficients of $R$ are $O(e^{-s})$ and
$R_\beta=O(e^{-s^2-s})$. A weighted cylinder estimate controls the nonlinear remainder. An integral identity then relates $A_\ell$ to this remainder. Combining the identity with the weighted estimate forces the exponential decay of $A_\ell$ in \eqref{mainrate}.

The construction was motivated by a reformulation of the normalization
equation using the $b$-calculus and the Mellin transform, informed by
Mazzeo's elliptic edge theory \cite{Mazzeo1991} and its continuation
with Boris Vertman \cite{MazzeoVertman2014}. This viewpoint brings
together radial decay and the angular spectrum, revealing why estimates
at fixed weights need not yield a single smooth flat solution.
The weighted estimates required here are proved directly below.

We prove Theorem~\ref{main} in the following section~\ref{sec:proof}. Finally, in section~\ref{sec:linear}, we identify the failure of the linear inverse underlying Cao's smooth existence argument. Using \eqref{mainrate}, we construct an analytic angular function $A$ whose coefficients decay too slowly for conformal normalization. The angular series of $g_A$ defining Cao's proposed inverse diverges, preventing the first iterate in the proof of Corollary~B from being well-defined.

\section{Proof of Theorem~\ref{main}}\label{sec:proof}
Fix $A$ and write $g_0=g_A$ throughout the proof. Set
$D=r\partial_r=-\partial_s$. We use $\partial_b^\alpha$ for a fixed
composition of $|\alpha|$ cylindrical derivatives, namely $\partial_s$
and smooth angular vector fields, taken in a fixed finite sphere atlas.

For $s\ge1$, elementary integration gives
\begin{equation}
0<F(s)\le\frac{e^{-s^2}}{2s},\qquad F'(s)=-e^{-s^2},\qquad
 |\partial_b^\alpha q|\le C_\alpha(1+s)^{|\alpha|}e^{-s^2}.\label{qder}
\end{equation}
In Cartesian coordinates,
\begin{equation}
\begin{aligned}
& (g_0)_{ij}
 =\delta_{ij}+(\kappa^{-1}-1)
       \left(\delta_{ij}-\frac{x_ix_j}{r^2}\right),\\
& |\partial_x^\alpha((g_0)_{ij}-\delta_{ij})|
 \le C_\alpha(1+s)^{|\alpha|}e^{-s^2+|\alpha|s}.
 \end{aligned}\label{cartmetric}
\end{equation}

Define
\begin{equation}
\begin{gathered}
 h=\kappa^{-1}h_0,\qquad J=\sqrt{\det((g_0)_{ij})}=e^q,\\
 dV_{g_0}=r^{n-1}J\,dr\,d\omega,\qquad
 B=D\ln J=Dq=-\partial_s q=e^{-s^2}A(\theta).
 \end{gathered}\label{B}
\end{equation}
  $B$ is the logarithmic radial
derivative of the ratio of the metric volume density to Euclidean volume
density in the fixed Cartesian chart.  We have
\[
 H_r-\frac1r=\frac{B}{(n-1)r}.
\]

The same argument as for \eqref{cartmetric} shows that $e^{2q}-1$
and the nonhomogeneous term $-r^{-2}e^{-s^2}A(\theta)$ extend smoothly
and flatly across $0$.

\subsection{The normalization equation}\label{sec:distance}
Let $r_1$ and $\Phi$ be as in Theorem~\ref{main}, and put
$g=\Phi g_0$. On a sufficiently small normal neighborhood, its distance
$\rho$ from $0$ has smooth square and satisfies
\begin{equation}\label{squareeikonal}
 |d(\rho^2)|_{g_0}^2=4\Phi\rho^2.
\end{equation}
We first show that $w=\ln(\rho/r)$ extends smoothly across $0$.
Write $(\rho^2)_j$ and $\Phi_j$ for the homogeneous Taylor terms of
degree $j$, and put $c=\Phi(0)>0$. Then $(\rho^2)_2=cr^2$.
Since $g_0-\delta$ is flat, the degree-$j$ part of
\eqref{squareeikonal}, for $j\ge3$, is
\[
 4c(j-1)(\rho^2)_j
 =4\sum_{i=1}^{j-2}\Phi_i(\rho^2)_{j-i}
  -\sum_{\substack{i,k\ge3\\i+k=j+2}}
       \ip{d(\rho^2)_i}{d(\rho^2)_k}_{\delta}.
\]
The gradient inner product of two polynomials divisible by $r^2$
is again divisible by $r^2$. Thus induction shows that every
$(\rho^2)_j$ is divisible by $r^2$. For every $M\ge0$, there is
therefore a polynomial $P_M$ of degree at most $M$ such that
\[
 \begin{gathered}
 j_0^{M+2}(\rho^2)=r^2P_M,\qquad
 \mathcal R_M=\rho^2-r^2P_M,\\
 \partial_x^\alpha\mathcal R_M=O(r^{M+3-|\alpha|}),\qquad
 \partial_x^\alpha(\mathcal R_M/r^2)=O(r^{M+1-|\alpha|})
 \quad(|\alpha|\le M).
 \end{gathered}
\]
The last remainder extends by zero as a $C^M$ function. Since $M$
is arbitrary, $\rho^2/r^2$ is smooth with value $c$ at $0$, and
\[
 w=\tfrac12\ln(\rho^2/r^2)\in C^\infty,\qquad \rho=re^w.
\]

\begin{remark}
At degree three, the recurrence gives
$8c(\rho^2)_3=4cr^2d\Phi(0)[x]$. Hence, for $M\ge1$,
\[
 \begin{gathered}
 P_M(0)=\Phi(0),\qquad dP_M(0)=\tfrac12d\Phi(0),
 \\
 w(0)=\tfrac12\ln\Phi(0),\qquad
 \beta:=dw(0)=\tfrac14d\ln\Phi(0).
 \end{gathered}
\]
The eikonal equation for $\rho=re^w$ also gives
\begin{equation}\label{Q}
 \Phi=|d\rho|_{g_0}^2=e^{2w}Q,\qquad
 Q=(1+Dw)^2+\kappa|d_Sw|_{h_0}^2.
\end{equation}
\end{remark}

\subsubsection{The scalar equation}\label{subsec:scalar}
The conformal transformation formula is
\begin{equation}\label{conflap}
 \Delta_g\rho=\Phi^{-1}\left(
 \Delta_{g_0}\rho+\frac{n-2}{2}
       \ip{d\ln\Phi}{d\rho}_{g_0}\right).
\end{equation}
Substituting \eqref{Q} into \eqref{conflap} and using \eqref{cn} yields
\begin{equation}\label{E}
 \begin{aligned}
 E_{g_0}(w)
 &:=B+D^2w+\Delta_hw+B\,Dw\\
 &\quad+\frac{n-2}{2Q}
     \left((1+Dw)DQ+\kappa\ip{d_SQ}{d_Sw}_{h_0}\right)=0.
 \end{aligned}
\end{equation}
This is Cao's scalar distance-quotient equation
\cite[p.~906, the displayed calculation preceding (1.14)]{Cao},
written in polar notation. It depends smoothly on the first and
second cylindrical derivatives of $w$ wherever $Q>0$, satisfies
$E_{g_0}(0)=B$, and has Euclidean linearization
\begin{equation}\label{N}
 N=E_\delta'(0)=(n-1)D^2+\Delta_S
                =(n-1)\partial_s^2+\Delta_S.
\end{equation}
Constant rescaling of $g$ preserves \eqref{cn}; we henceforth
normalize $\Phi(0)=1$, so that $w(0)=0$.

\subsection{The Euclidean M\"obius solution}\label{sec:nonlinear}
With the Euclidean identification of $\beta$ and its dual vector, set
\[
 W_\beta=1-2\beta\cdot x+|\beta|^2r^2,\qquad
 T_\beta(x)=\frac{x-\beta r^2}{W_\beta(x)}.
\]
This is the composition $I\circ(I-\beta)$, where $I(x)=x/|x|^2$,
and it extends smoothly across $0$. Its pullback metric and distance
from $0$ are
\[
 T_\beta^*\delta=W_\beta^{-2}\delta,\qquad
 \rho_\beta=|T_\beta(x)|=\frac r{\sqrt{W_\beta}},\qquad
 w_\beta=-\tfrac12\ln W_\beta.
\]
Since $T_\beta$ is a local isometry to Euclidean space,
$E_\delta(w_\beta)=0$. Also $w_\beta(0)=0$ and $dw_\beta(0)=\beta$,
so $v:=w-w_\beta$ has zero first jet.

We claim that $v$ is flat. If its terms of degrees less than $j\ge2$
vanish, let $v_j$ be its homogeneous degree-$j$ term. The first and
second cylindrical derivatives of $w_\beta+\tau v$ are $O(r)$, so
\eqref{E} shows that $E_\delta'(w_\beta+\tau v)-N$ has order at most
two and coefficients $O(r)$, uniformly for $0\le\tau\le1$.
Since $g_0-\delta$ is flat, it contributes no Taylor terms.
Consequently, the degree-$j$ part of
$E_{g_0}(w)-E_\delta(w_\beta)=0$ is $Nv_j=0$.
The restriction of a homogeneous polynomial of degree $j$ to $S$
contains only spherical harmonic degrees $j,j-2,\ldots$. Its sphere
Dirichlet energy is therefore at most $j(j+n-2)$ times its squared
$L^2$ norm. Restricting all quantities below to $S$, we obtain
\begin{equation}\label{formal}
 \begin{aligned}
 0=\ip{Nv_j}{v_j}_{L^2(S)}
  &=(n-1)j^2\|v_j\|_{L^2(S)}^2-\|d_Sv_j\|_{L^2(S)}^2\\
  &\ge(n-2)j(j-1)\|v_j\|_{L^2(S)}^2.
 \end{aligned}
\end{equation}
Thus $v_j=0$, and induction proves the claim.

\subsection{The nonlinear error estimates}\label{subsec:error}
Put $\mathcal M=1+\|A\|_{C^1(S)}+|\beta|+|\beta|^2$.
Choose $0<r_0<\min(r_1,e^{-1})$ so that all preceding constructions are defined
on $\overline B_{r_0}(0)$ and
\begin{equation}\label{radiuschoice}
 \|v\|_{C^2(\overline B_{r_0}(0))}\le1,\qquad
 r_0\mathcal M\le\frac1{16}.
\end{equation}
Here the $C^2$ norm uses Euclidean tensor norms. This choice is possible because $v$ is smooth and flat. Write
$s_0=-\ln r_0$. Define
\[
 R_\beta=E_{g_0}(w_\beta)-B,\qquad
 w_\tau=w_\beta+\tau v\quad(0\le\tau\le1),
\]
and, for a smooth test function $z$ on the cylinder, set
\begin{equation}\label{P}
 \begin{gathered}
 E_{g_0}'(w)z
 =\left.\frac{d}{d\lambda}E_{g_0}(w+\lambda z)\right|_{\lambda=0},\\
 Pz=\int_0^1E_{g_0}'(w_\tau)z\dd\tau,\qquad R=P-N.
 \end{gathered}
\end{equation}
The fundamental theorem of calculus gives
$Pv=E_{g_0}(w)-E_{g_0}(w_\beta)$, and hence
\begin{equation}\label{Pv}
 (N+R)v=-B-R_\beta.
\end{equation}

\begin{lemma}\label{lem:nonlinear-errors}
There is a constant $D_n$, depending only on $n$, such that,
for $s\ge s_0$,
\begin{equation}\label{residual}
 \|R_\beta(s,\cdot)\|_{L^2(S)}
 \le D_n\mathcal M^2e^{-s^2-s}.
\end{equation}
For every smooth function $z$ on the cylinder,
\begin{equation}\label{Rbound}
 \|Rz(s,\cdot)\|_{L^2(S)}
 \le D_n\mathcal M e^{-s}
       \sum_{1\le j+k\le2}
       \|\partial_s^j\nabla_S^kz(s,\cdot)\|_{L^2(S)}.
\end{equation}
\end{lemma}

\begin{proof}
Let
\[
 \mathcal J(w)=(Dw,d_Sw,D^2w,d_SDw,\nabla_S^2w),
 \qquad K=(q,B,d_Sq),
\]
with the sum of the round tensor norms on each tuple. These are
pointwise finite-dimensional variables; $\mathcal J$ contains no
zeroth-order term. Differentiating the explicit $W_\beta$ gives
$|\mathcal J(W_\beta)|\le C_n(|\beta|r+|\beta|^2r^2)$.
Since $W_\beta\ge(1-|\beta|r)^2$ and $|\beta|r\le1/16$,
the chain rule for $w_\beta=-\frac12\ln W_\beta$ yields
\[
 |\mathcal J(w_\beta)|\le C_n|\beta|r.
\]
The zero first jet of $v$ and \eqref{radiuschoice} imply
$|d_xv|\le r$ and $|\nabla_x^2v|\le1$. Cylindrical differentiation
therefore gives $|\mathcal J(v)|\le C_nr^2$. Together with
\eqref{qder} and $B=e^{-s^2}A$, these estimates imply
\[
 |\mathcal J(w_\tau)|\le C_n\mathcal M e^{-s},\qquad
 |K|\le C_n\|A\|_{C^1(S)}e^{-s^2}.
\]
Moreover, uniformly in $\tau$,
\[
 |Dw_\tau|
 \le\frac{|\beta|r(1+|\beta|r)}{(1-|\beta|r)^2}+r^2<\frac18.
\]
Thus $Q(w_\tau)\ge(1+Dw_\tau)^2\ge1/2$. The same lower bound
holds if $q$ is replaced by any multiple in $[0,1]$ and
$\mathcal J(w_\tau)$ by any multiple in $[0,1]$.

To make the coefficient dependence explicit, for a generic $w$ put
$b=1+Dw$ and $p=\nabla_Sw$, so that $Q=b^2+\kappa|p|^2$.
Using
\[
 \Delta_hw=\kappa\Delta_Sw+\frac{n-3}{n-1}
                 \kappa\ip{d_Sq}{d_Sw}_{h_0},\qquad
 D\kappa=-\frac{2\kappa B}{n-1},\qquad
 d_S\kappa=-\frac{2\kappa}{n-1}d_Sq,
\]
and differentiating $Q$ in \eqref{E}, we obtain
\[
 \begin{aligned}
 E_{g_0}(w)-B={}&
 \left(1+\frac{(n-2)b^2}{Q}\right)D^2w
 +\frac{2(n-2)\kappa b}{Q}\ip{p}{\nabla_SDw}_{h_0}\\
 &+\kappa\Delta_Sw
 +\frac{(n-2)\kappa^2}{Q}\nabla_S^2w(p,p)
 +B\,Dw+\frac{n-3}{n-1}\kappa\ip{\nabla_Sq}{p}_{h_0}\\
 &-\frac{n-2}{n-1}\frac{\kappa}{Q}
       \bigl(bB+\kappa\ip{\nabla_Sq}{p}_{h_0}\bigr)|p|^2.
 \end{aligned}
\]
This expression defines a smooth function
$\mathcal F(K,V)$ with $V=\mathcal J(w)$ and $\kappa=e^{-2q/(n-1)}$.
The preceding bounds place its arguments and the segments joining
them to zero in a fixed bounded set on which $Q\ge1/2$.
All first and second derivatives of $\mathcal F$ in these
finite-dimensional variables are consequently bounded by constants
depending only on $n$. The displayed formula gives
\[
 \mathcal F(K,0)=0,\qquad
 \mathcal F(0,\mathcal J(w))=E_\delta(w),\qquad
 D_V\mathcal F(0,0)\mathcal J(z)=Nz.
\]
Since $D_K\mathcal F(K,0)=0$, two applications of the fundamental
theorem of calculus, first in $K$ and then in $V$, give
\[
 |\mathcal F(K,V)-\mathcal F(0,V)|\le C_n|K||V|.
\]
As $E_\delta(w_\beta)=0$, it follows that
\[
 |R_\beta|
 =|\mathcal F(K,\mathcal J(w_\beta))
      -\mathcal F(0,\mathcal J(w_\beta))|
 \le C_n\mathcal M^2e^{-s^2-s}.
\]
Taking the sphere norm proves \eqref{residual}.

Likewise, the mean-value estimate for $D_V\mathcal F$ gives
\[
 \begin{aligned}
 |(E_{g_0}'(w_\tau)-N)z|
 &\le C_n\bigl(|K|+|\mathcal J(w_\tau)|\bigr)|\mathcal J(z)|\\
 &\le C_n\mathcal M e^{-s}|\mathcal J(z)|.
 \end{aligned}
\]
This also bounds every coefficient of $E_{g_0}'(w_\tau)-N$:
it has order at most two, has no zeroth-order term, and its
coefficients in a fixed finite sphere atlas are bounded by
$C_n\mathcal M e^{-s}$. Integrating in $\tau$, then taking the sphere
$L^2$ norm and using the triangle inequality, proves \eqref{Rbound}
and the same coefficient bounds for $R$.
\end{proof}

\subsection{The weighted cylinder estimate}\label{sec:weighted}
Set
\[
 \lambda_k=k(k+n-2),\qquad
 \mu_k=\sqrt{\lambda_k/(n-1)}.
\]
The function $f(x)=\sqrt{x(x+n-2)/(n-1)}$ satisfies
$f'(x)\ge1/\sqrt{n-1}$ for $x>0$, and $\mu_1-\mu_0=1$.
Thus $\mu_{k+1}-\mu_k\ge1/\sqrt{n-1}$. For every sufficiently
large degree $\ell$, choose
\begin{equation}\label{weight}
 a=\frac1{4\sqrt{n-1}},\qquad
 \sigma=\mu_\ell-a\ge1,\qquad
 \dist(\sigma,\{\mu_k:k\ge0\})\ge a.
\end{equation}
Indeed, the distance to $\mu_\ell$ is $a$, whereas the distance to
any other root is at least $3a$. In particular, $0<a<1$.

In this and the next subsection, unlabelled $L^2$ norms are taken on
$\R\times S$ with measure $ds\,d\omega$. Define the weighted
derivative norm
\[
 \mathcal D_\sigma(z)^2
 =\sum_{j+k\le2}
   \|e^{\sigma s}\partial_s^j\nabla_S^kz\|_{L^2}^2.
\]
We claim that, uniformly for the weights in \eqref{weight},
\begin{equation}\label{Nest}
 \mathcal D_\sigma(z)
 \le C_n\sigma\|e^{\sigma s}Nz\|_{L^2}.
\end{equation}
Here and below, constants $C_n$ depend only on $n$ and may increase
from one occurrence to the next.

To prove \eqref{Nest}, first take $z$ smooth and compactly supported,
put $Z=e^{\sigma s}z$, and take the unitary Fourier transform in $s$.
The conjugated operator
\[
 N_\sigma=e^{\sigma s}Ne^{-\sigma s}
 =(n-1)(\partial_s-\sigma)^2+\Delta_S
\]
has multiplier
\[
 p_\sigma(\xi,k)
 =(n-1)(i\xi-\sigma-\mu_k)(i\xi-\sigma+\mu_k).
\]
Writing $d_\pm=\sqrt{\xi^2+(\mu_k\pm\sigma)^2}$, we have
$d_-\ge a$ and $d_+\le d_-+2\sigma$. Consequently,
\[
 \frac{\sigma^2+\xi^2+\lambda_k}{|p_\sigma(\xi,k)|}
 \le\frac{d_+}{d_-}
 \le1+\frac{2\sigma}{a}
 \le C_n\sigma.
\]
The radial derivatives of $z$ become powers of $i\xi-\sigma$ after
conjugation. For the angular derivatives, integration by parts and
the integrated Bochner identity on the round sphere give
\[
 \|\nabla_Sf\|_2^2=\langle-\Delta_Sf,f\rangle,
 \qquad
 \|\nabla_S^2f\|_2^2
 =\|\Delta_Sf\|_2^2-(n-2)\|\nabla_Sf\|_2^2
 \le\|\Delta_Sf\|_2^2.
\]
Each derivative term in $\mathcal D_\sigma(z)$ is therefore bounded
by the multiplier $\sigma^2+\xi^2+\lambda_k$. Plancherel yields
\[
 \begin{aligned}
 \mathcal D_\sigma(z)^2
 &\le C_n\sum_{k\ge0}\int_\R
       (\sigma^2+\xi^2+\lambda_k)^2
       \|\widehat{\Pi_kZ}(\xi)\|_{L^2(S)}^2\dd\xi\\
 &\le C_n\sigma^2\|N_\sigma Z\|_{L^2}^2,
 \end{aligned}
\]
which proves \eqref{Nest}. Density extends it to every $z$ with
$\mathcal D_\sigma(z)<\infty$.

Choose a fixed smooth function $\chi$, with $0\le\chi\le1$, which is
zero on $(-\infty,0]$ and one on $[1,\infty)$. Set
\[
 L=\frac\sigma8,\qquad
 \eta(s)=\chi(s-L),\qquad z=\eta v.
\]
We require $\sigma\ge16s_0$ and increase its dimension-dependent
lower bound as needed below. Since $s_0\ge1$, this gives
$L\ge s_0+1$. Define $z=0$ for $s\le L$; no extension of $v$ or
of the coefficients of $R$ is involved. Flatness of $v$ implies
$\mathcal D_\sigma(z)<\infty$ for each fixed $\sigma$: choose a
fixed decay order greater than $\sigma$ for its first two
cylindrical derivatives. Thus \eqref{Nest} applies to $z$.

Using $N=P-R$ and \eqref{Rbound}, we obtain directly
\[
 \mathcal D_\sigma(z)
 \le C_n\sigma\|e^{\sigma s}Pz\|_{L^2}
       +C_n\mathcal M\sigma e^{-L}\mathcal D_\sigma(z).
\]
Both $Pz$ and $Rz$ are used only on $s\ge L$ and are zero below $L$.
By \eqref{radiuschoice},
$\mathcal M e^{-s_0}\le1/16$, so
\[
 \mathcal M\sigma e^{-L}
 \le\frac\sigma{16}e^{s_0-\sigma/8}
 \le\frac\sigma{16}e^{-\sigma/16}.
\]
For sufficiently large $\sigma$, depending only on $n$, this
absorbs the last term and proves
\[
 \mathcal D_\sigma(z)
 \le C_n\sigma\|e^{\sigma s}Pz\|_{L^2}.
\]

Since $Pv=-B-R_\beta$, the cutoff equation is
\begin{equation}\label{Pcut}
 Pz=-\eta B-\eta R_\beta+c_\sigma,\qquad
 c_\sigma=[P,\eta]v,\qquad
 \supp c_\sigma\subset[L,L+1].
\end{equation}
The commutator contains only $v$ and its first derivatives,
multiplied by bounded coefficients of $P$ and derivatives of $\chi$.
Indeed, a radial second derivative contributes
$\eta''v+2\eta'\partial_sv$, a mixed derivative contributes
$\eta'\partial_Av$, and a radial first derivative contributes
$\eta'v$. The bounds in \eqref{radiuschoice} and \eqref{Rbound}
therefore give
\[
 \|c_\sigma(s)\|_{L^2(S)}\le C_n\mathcal M.
\]
For the sources, completing the square gives
\[
 \int_L^\infty e^{2\sigma s-2s^2}\dd s
 \le\sqrt{\pi/2}\,e^{\sigma^2/2}.
\]
Together with \eqref{B} and \eqref{residual}, this implies
\[
 \begin{aligned}
 \|e^{\sigma s}\eta B\|_{L^2}
  +\|e^{\sigma s}\eta R_\beta\|_{L^2}
 &\le C_n\mathcal M^2e^{\sigma^2/4},\\
 \|e^{\sigma s}c_\sigma\|_{L^2}
 &\le C_n\mathcal M e^{\sigma(L+1)}.
 \end{aligned}
\]
For $\sigma\ge8$, our choice of $L$ gives
$\sigma(L+1)=\sigma^2/8+\sigma\le\sigma^2/4$.
We conclude that
\begin{equation}\label{weightedz}
 \mathcal D_\sigma(z)
 \le C_n\mathcal M^2\sigma e^{\sigma^2/4}.
\end{equation}

\subsection{The integral identity and coefficient estimate}\label{sec:moments}
Fix a sufficiently large degree $\ell$, write $\mu=\mu_\ell$, and
retain $\sigma=\mu-a$ and $L=\sigma/8$. Projecting \eqref{Pcut}
onto degree $\ell$, with $z_\ell=\Pi_\ell z$, gives
\[
 (n-1)(z_\ell''-\mu^2z_\ell)
 =-\eta(s)e^{-s^2}A_\ell
       +\Pi_\ell(-\eta R_\beta-Rz+c_\sigma).
\]
Multiply by $e^{\mu s}$ and integrate from $L$ to infinity. The
left side integrates to
\[
 (n-1)\left[e^{\mu s}(z_\ell'-\mu z_\ell)\right]_L^\infty=0.
\]
At $L$, both $z$ and $\partial_sz$ vanish by the cutoff. At infinity,
flatness gives $v,\partial_sv=O(e^{-Ns})$ for every fixed $N$;
choosing $N>\mu$ proves the endpoint vanishes for this fixed degree.
The integrals converge for the same reason. Hence
\begin{equation}\label{moment}
 \begin{aligned}
 I_\ell A_\ell
 &=\int_L^\infty e^{\mu s}
        \Pi_\ell(-\eta R_\beta-Rz+c_\sigma)\dd s,\\
 I_\ell&=\int_L^\infty\eta(s)e^{\mu s-s^2}\dd s.
 \end{aligned}
\end{equation}
This is an identity in the degree-$\ell$ harmonic space; the cutoff
does not impose boundary data on the original function $v$.

We estimate its terms before dividing by $I_\ell$. Since
$\mu=\sigma+a$, \eqref{Rbound} and Cauchy--Schwarz give
\[
 \begin{aligned}
 \int_L^\infty e^{\mu s}\|Rz(s)\|_{L^2(S)}\dd s
 &\le C_n\mathcal M
   \left(\int_L^\infty e^{-2(1-a)s}\dd s\right)^{1/2}
   \mathcal D_\sigma(z)\\
 &\le C_n\mathcal M^3\sigma e^{\sigma^2/4}.
 \end{aligned}
\]
The scalar integral is uniformly bounded because $a<1$ and $L>0$;
the finite sum of derivatives is absorbed into $C_n$.
For the residual and commutator, \eqref{residual} and their support
properties yield
\[
 \begin{aligned}
 \int_L^\infty e^{\mu s}\|\eta R_\beta(s)\|_{L^2(S)}\dd s
 &\le C_n\mathcal M^2
       \int_L^\infty e^{(\mu-1)s-s^2}\dd s\\
 &\le C_n\mathcal M^2e^{(\mu-1)^2/4}
 \le C_n\mathcal M^2e^{\sigma^2/4},\\
 \int_L^\infty e^{\mu s}\|c_\sigma(s)\|_{L^2(S)}\dd s
 &\le C_n\mathcal M e^{\mu(L+1)}.
 \end{aligned}
\]
Here $0\le\mu-1<\sigma$ for the degrees under consideration.

For sufficiently large $\mu$, the interval
$[\mu/2-1,\mu/2+1]$ lies beyond $L+1$, where $\eta=1$.
Another square completion therefore gives
\[
 I_\ell\ge
 \int_{\mu/2-1}^{\mu/2+1}e^{\mu s-s^2}\dd s
 =c_0e^{\mu^2/4},\qquad
 c_0=\int_{-1}^1e^{-y^2}\dd y>0.
\]
Since $\Pi_\ell$ is an orthogonal projection, \eqref{moment} now
implies
\[
 \begin{aligned}
 \|A_\ell\|_{L^2(S)}
 &\le C_n\mathcal M^3
   \left((1+\sigma)e^{(\sigma^2-\mu^2)/4}
             +e^{\mu(L+1)-\mu^2/4}\right)\\
 &\le C_n\mathcal M^3
   \left((1+\mu)e^{-a\mu/2+a^2/4}
             +e^{-\mu^2/16}\right)
 \le C_n\mathcal M^3e^{-a\mu/4}.
 \end{aligned}
\]
In the second line we used $L=(\mu-a)/8$ and $\mu\ge16$; in the
last line, the polynomial factor is absorbed into half of the
linear exponential decay. Since $\mu_\ell\ge\ell/\sqrt{n-1}$,
this proves the required rate
$a/(4\sqrt{n-1})=c_n=1/[16(n-1)]$ for all sufficiently large degrees.

To record the constant dependence, choose $K_n$ sufficiently large,
depending only on $n$, and set
\[
 \ell_0=\left\lceil K_n(1+s_0)\right\rceil.
\]
For $\ell\ge\ell_0$, the bound
$\sigma\ge\ell/\sqrt{n-1}-a$ ensures all the conditions above:
$\sigma\ge16s_0$, the fixed lower bound needed for absorption,
and the numerical conditions for the cutoff and Gaussian integrals.
For $0\le\ell<\ell_0$, orthogonal projection gives
\[
 e^{c_n\ell}\|A_\ell\|_{L^2(S)}
 \le e^{c_n\ell_0}\|A\|_{L^2(S)}.
\]
Thus, with
\[
 C=C_n\mathcal M^3+e^{c_n\ell_0}\|A\|_{L^2(S)},
\]
we have
\begin{equation}\label{Adecay}
 \|A_\ell\|_{L^2(S)}\le Ce^{-c_n\ell}\qquad(\ell\ge0).
\end{equation}
Because $s_0=-\ln r_0$, $\mathcal M\le(16r_0)^{-1}$, and
$\|A\|_{L^2(S)}\le|S|^{1/2}\|A\|_{C^1(S)}$, this constant has
the dependence asserted in Theorem~\ref{main}. The selected radius
may depend on the full conformal normalizer; once it is fixed, no
higher derivative norm or additional dependence on the original
value of $\Phi(0)$ enters the estimate.

Since the dimensions of the spherical harmonic spaces grow at most
polynomially, \eqref{Adecay} implies that the homogeneous harmonic
extensions $H_\ell(r\theta)=r^\ell A_\ell(\theta)$ have a locally
uniformly convergent sum on $|x|<e^{c_n}$.
This sum is harmonic and hence real analytic, and its restriction
to $S$ is $A$, completing the proof of Theorem~\ref{main}.

\section{The failure of Cao's linear inverse argument}\label{sec:linear}
For each $\ell\ge1$, choose a real spherical harmonic $\psi_\ell$ of
degree $\ell$ with $\|\psi_\ell\|_{L^2(S)}=1$, and define
\begin{equation}
A(\theta)=\sum_{\ell=1}^{\infty}
                   e^{-c_n\ell/2}\psi_\ell(\theta).\label{introanalytic}
\end{equation}
The function \eqref{introanalytic} is real analytic on $S$, but
the metric $g_A$ admits no smooth conformal normal representative
near $0$.

For Cao's scaling linearization
\[
 \mathscr L_0=\Delta_{\R^n}+(n-2)\partial_r^2,
\]
consider the following non-homogeneous term
\begin{equation}\label{linearforcing}
 f(r,\theta)=-r^{-2}e^{-(-\ln r)^2}A(\theta).
\end{equation}
Write
$a_\ell=e^{-c_n\ell/2}$ and
$\mu_\ell=\sqrt{\ell(\ell+n-2)/(n-1)}$.

Cao's space $Z_\rho$ is the intersection of the weighted spaces
over all differentiation orders and all decay orders
\cite[(2.2)--(2.4)]{Cao}. A smooth flat function belongs to these
spaces after restricting to a smaller ball. In particular, the function
\eqref{linearforcing} belongs to the domain of the right inverse
claimed in his Corollary~2.14.

In the notation of \cite[(2.18)--(2.19)]{Cao}, the indicial roots
for degree $\ell\ge1$ are
\[
 \gamma_\ell=\mu_\ell,\qquad \eta_\ell=-\mu_\ell.
\]
We index by degree rather than by an enumeration with multiplicity;
all eigenfunctions of the same degree have the same endpoints.
For a fixed weight $N>1$, his choice on p.~912 becomes
\[
 c_\ell^{(N)}=
 \begin{cases}0,&\mu_\ell\le N,\\1,&\mu_\ell>N,\end{cases}
 \qquad d_\ell^{(N)}=0.
\]
Indeed, $\eta_\ell\le0<N-1$. Equivalently, the change from zero
to one occurs after degree
\[
 L_N=\left\lfloor
       \frac{\sqrt{(n-2)^2+4(n-1)N^2}-(n-2)}2
       \right\rfloor.
\]
The endpoint $1$ is the outer radius of his normalized unit ball.
For the coefficient $f_\ell(r)=-a_\ell r^{-2}e^{-(-\ln r)^2}$,
the radial formula is
\[
 \begin{split}
 \beta_\ell^{(N)}(r)=\frac1{(2n-2)\mu_\ell}\bigg[
 &r^{\mu_\ell}\int_{c_\ell^{(N)}}^r
               f_\ell(t)t^{1-\mu_\ell}\dd t\\
 &-r^{-\mu_\ell}\int_0^r
               f_\ell(t)t^{1+\mu_\ell}\dd t\bigg].
 \end{split}
\]
In the proof of Corollary~2.14 on p.~914, Cao instead sets
\begin{equation}\label{flat-endpoints}
 c_\ell=d_\ell=0\quad\hbox{for every degree}.
\end{equation}
For each individual degree, both integrals in this formula converge.
For $\ell>L_N$, the change of endpoint adds
\begin{equation}\label{endpoint-difference}
 \begin{split}
 \beta_\ell^{\mathrm{flat}}(r)-\beta_\ell^{(N)}(r)
 &=\frac{r^{\mu_\ell}}{(2n-2)\mu_\ell}
       \int_0^1 f_\ell(t)t^{1-\mu_\ell}\dd t\\
 &=-\frac{a_\ell r^{\mu_\ell}}{(2n-2)\mu_\ell}e^{\mu_\ell^2/4}
       \int_0^\infty e^{-(y-\mu_\ell/2)^2}\dd y.
 \end{split}
\end{equation}
The constant mode causes no difficulty: $A_0=f_0=0$, and its
flat solution is zero.

The last paragraph on p.~914 asserts that the earlier estimates
continue to hold with \eqref{flat-endpoints}. For the function
\eqref{linearforcing}, this assertion fails, as the following
calculation shows.

In polar coordinates,
\[
 r^2\mathscr L_0=(n-1)(r\partial_r)^2+\Delta_S
              =(n-1)\partial_s^2+\Delta_S.
\]
Suppose that $\mathscr L_0v=f$ has a smooth flat solution on
$0\le r<r_{\mathrm{lin}}$.
For $s> -\ln r_{\mathrm{lin}}$, define its scalar coefficient
\[
 v_\ell(s)=\int_S v(e^{-s},\theta)\psi_\ell(\theta)\dd\omega.
\]
Projection of the equation gives
\[
 (n-1)(v_\ell''-\mu_\ell^2v_\ell)=-a_\ell e^{-s^2}.
\]
Its solution flat at $r=0$ is necessarily
\begin{equation}\label{flat-mode}
 v_\ell(s)=-\frac{a_\ell}{(n-1)\mu_\ell}
       \int_s^\infty\sinh\bigl(\mu_\ell(y-s)\bigr)e^{-y^2}\dd y.
\end{equation}
To verify the formula, denote its integral by $K_\mu(s)$.
Differentiation under the integral gives
$K_\mu''-\mu^2K_\mu=\mu e^{-s^2}$.
For fixed $\mu$, the substitution $y=s+t$ bounds the integral and
each fixed derivative by a polynomial in $s$ times $e^{-s^2}$
for sufficiently large $s$; for example,
\[
 0\le K_\mu(s)\le
 e^{-s^2}\int_0^\infty e^{-(2s-\mu)t}\dd t
 =\frac{e^{-s^2}}{2s-\mu}\quad(2s>\mu).
\]
Thus \eqref{flat-mode} is flat. Any other solution differs by
$C_\ell e^{\mu_\ell s}+D_\ell e^{-\mu_\ell s}$, and flatness
forces both constants to vanish.

Now fix any sufficiently large but finite $s$ and let $\mu\to\infty$.
Completing the square gives
\begin{align*}
 K_\mu(s)
 &=\frac12e^{\mu^2/4-\mu s}
         \int_s^\infty e^{-(y-\mu/2)^2}\dd y
       -\frac12e^{\mu s}\int_s^\infty e^{-y^2-\mu y}\dd y.
\end{align*}
The first integral tends to $\sqrt\pi$. For the second term,
put $y=s+t$ to obtain the bound
$e^{-s^2}/[2(\mu+2s)]$ when $\mu+2s>0$.
Consequently
\begin{equation}\label{linear-growth}
 |v_\ell(s)|\sim\frac{\sqrt\pi}{(2n-2)\mu_\ell}
  \exp\left(\frac{\mu_\ell^2}{4}-\mu_\ell s-\frac{c_n\ell}{2}\right).
\end{equation}
Since $(n-1)\mu_\ell^2=\ell(\ell+n-2)$, the positive quadratic
term in $\ell$ dominates the negative linear terms. Thus these
coefficients tend to infinity. On the other hand, the functions
$\psi_\ell$ form an orthonormal sequence, so the coefficients of
$v(e^{-s},\cdot)\in L^2(S)$ must tend to zero by Bessel's inequality.
Thus no smooth flat solution exists on any neighborhood of $0$.

The coefficients produced by the zero endpoints are precisely
\eqref{flat-mode}, so their angular series is not an $L^2(S)$
function at any fixed small radius. The radial cutoff in Cao's definition
of the inverse is identically one near $0$ and cannot change this conclusion.

For our metric, \eqref{B} gives exactly the right-hand side in
Cao's equation $(**)$ on p.~915:
\[
 -\frac1r\partial_r\ln\sqrt{\det(g_A)}
 =-r^{-2}e^{-(-\ln r)^2}A(\theta)=f(r,\theta).
\]
His proof of Corollary~B on pp.~917--918 uses the inverse
$\mathscr S$ of Corollary~2.14, beginning with
$\bar u_{1,\rho}=\mathscr S[f(\rho\,\cdot)]$.
This first iterate is not defined by the proposed series for $f(r,\theta).$  
Thus the analytic angular function \eqref{introanalytic} shows both
the nonlinear nonexistence established by Theorem~\ref{main} and the failure of the proposed linear inverse at the first step of Cao's
iteration. The failed step is the passage from weight-dependent
integration endpoints to one inverse on all decay orders: convergence
for each angular degree does not imply convergence of their sum.

These examples raise a general geometric question and a more
specific spectral question.

\begin{question}\label{general-question}
Given a smooth Riemannian metric $g_0$ near a point $p$ in dimension
$n\ge3$, what necessary-and-sufficient geometric condition guarantees
the existence of a smooth positive function $\Phi$ near $p$ for which
$\Phi g_0$ is conformal normal at $p$?
\end{question}

The calculation also suggests a specific relation between radial
and angular regularity. For $p>1$, consider
\[
 F_p(s)=\int_s^\infty e^{-t^p}\dd t,\qquad
 g_{p,A}=dr^2+r^2e^{2F_p(s)A(\theta)/(n-1)}h_0,
 \qquad s=-\ln r.
\]
These are flat metrics at $0$. In the linear equation, radial decay competes with angular frequency through the exponent $\lambda s-s^p$.
Its maximum is the Legendre transform
\[
 \omega_p(\lambda)
 =(p-1)\left(\frac{\lambda}{p}\right)^{p/(p-1)}.
\]
Thus the linear calculation suggests a corresponding angular
decay rate $e^{-\omega_p(\lambda)}$. In particular,
$\omega_2(\lambda)=\lambda^2/4$.

\begin{question}\label{profile-question}
Does this relation between radial and angular decay characterize
nonlinear conformal normalization? More precisely, with
\[
 \Lambda=\sqrt{\frac{-\Delta_S}{n-1}},
\]
does $g_{p,A}$ admit a smooth conformal normal representative
near $0$ if and only if
\[
 e^{\omega_p(\Lambda)-b\Lambda}A\in L^2(S)
 \qquad\text{for some }b>0?
\]
\end{question}
\bibliographystyle{amsplain}
\bibliography{references}
\end{document}